\documentclass[12pt]{article}
\usepackage[T2A]{fontenc}
\usepackage[russian,english]{babel}
\usepackage[a4paper,mag=1000,includefoot,left=2cm,right=2cm,top=2cm,bottom=2cm,headsep=1cm,footskip=1cm]{geometry}
\usepackage{amsfonts,amssymb,amsmath,amsthm,indentfirst}
\usepackage{bm}
\usepackage{cite}
\usepackage{graphicx}

\newtheorem{proposition}{\indent Proposition}
\newtheorem{remark}{\indent Remark}

\def\ve{\varepsilon}

\def\leq{\leqslant}
\def\geq{\geqslant}

\def\*#1{\mathbf{#1}}

\begin{document}
\centerline{\large\textbf{A Semi-Explicit Compact Fourth-Order Finite-Difference Scheme}}
\centerline{\large\textbf{for the General Acoustic Wave Equation}
}
\bigskip\centerline{A. Zlotnik, T. Lomonosov}
\smallskip
\normalsize
\centerline{\textit{
Higher School of Economics University, Moscow, 109028 Pokrovskii bd. 11 Russia}}
\smallskip
\centerline{\it e-mail: azlotnik@hse.ru, tlomonosov@hse.ru}
\smallskip

\begin{abstract}
\noindent We construct a new compact semi-explicit three-level in time fourth-order finite-difference scheme for numerical solving the general multidimensional acoustic wave equation, where both the speed of sound and density of a medium are variable. 
The scheme is three-point in each spatial direction, has the truncation order $\mathcal{O}(|h|^4+h_t^4)$ and is easily implementable.
It seems to be the first compact scheme with such properties for the equation under consideration. 
It generalizes a semi-explicit compact scheme developed and studied recently in the much simpler case of the variable speed of sound only.
Numerical experiments confirm the high precision of the scheme and its fourth error order not only in the mesh $C$ norm but in the mesh $C^1$ norm as well.

\end{abstract}

\section{\normalsize Introduction}
\label{sec:intro}

The acoustic wave equation with the variable speed of sound $c(x)$ and density $\sigma(x)$ of the medium is important in some physical and engineering applications, for example, see \cite{BG12}.
We construct a new compact semi-explicit three-level in time fourth-order finite-difference scheme for solving such general $n$-dimensional acoustic wave equation, $n\geq 1$.
The scheme is three-point in each spatial direction and has the truncation order $\mathcal{O}(|h|^4+h_t^4)$.
It seems to be the first compact scheme with such properties for the equation under consideration. 

\par Higher-order compact schemes of several types in the much simpler case where only $c(x)$ is variable but $\sigma={\rm const}$ have recently been studied, in particular, see 
\cite{BTT18,CL75,CHFPSH23,DLG14,LLL19,Liao14,LYDH18,STT19,ZC22,ZC23} and references therein. 
A lot of papers on higher-order compact schemes were devoted also to the case of wave equations with constant coefficients that we almost do not touch here.
Some methods of other types to treat numerically the general acoustic wave equation were considered, in particular, see  \cite{BDE21,CFHJSS18,С02,CJ96,LLL23,SWK18} and references therein, but they are beyond the scope of this paper.
Of course, both lists do not pretend to be complete.

\par The new scheme generalizes a semi-explicit compact scheme developed and studied recently in the particular cases of constant $c$ and $\sigma$ and the variable $c=c(x)$ but $\sigma={\rm const}$
\cite{JG20,ZArxiv21,JG23,ZL24,ZL25}. 
The specific feature of the proposed scheme is involving of $n$ auxiliary unknown functions which approximate $n$ summands of the spatial part of the acoustic wave equation in each spatial direction.
In our generalization, an application of the three-point fourth order Samarskii scheme for the second order ordinary differential equation (ODE) in divergent form with a variable coefficient \cite{S63} is essential; this scheme generalizes the well-known Numerov scheme in the case of the constant coefficient.
We also suggest a modification of the Samarskii scheme to ensure better algebraic properties such as the diagonal dominance and positive definiteness for the involved three-point operator connected to the free term in the equation while maintaining  the fourth truncation order.
Notice that our scheme including its initial conditions does not contain derivatives of the free term and initial data of the problem that allows one to apply the scheme in the case where they are nonsmooth like in \cite{ZL25}.

\par The constructed scheme is conditionally stable as any other known higher-order three-level in time compact scheme for the wave-type equations that is three-point in each spatial direction.
The scheme can be easily implemented and requires to solve only independent tridiagonal systems of linear algebraic equations in each spatial direction (that can be accomplished in parallel).

\par We present results of the 2D numerical experiments that confirm the high precision of the scheme even for rough meshes and its fourth error order not only in the mesh $C$ (i.e., uniform) norm but in the mesh $C^1$ norm as well.
Such properties in the latter norm are important for accurate uniform computation of some additional physical quantities but have previously not been analyzed. 
We consider examples with smooth $\sigma(x)$ and $c(x)$ and with $\sigma(x)$ and $c(x)$ having a smoothed jump very steep in the case of $c(x)$. 
Note that, in our computations, we observe the possibility of using larger Courant numbers with respect to the variable $\sigma$ than those predicted theoretically.
In addition, we include a study of the acoustic wave propagation in the three-layer-type medium, with $\sigma(x)$ and $c(x)$ having steep smoothed jumps,  generated by a Ricker-type wavelet source function smoothed in space. 
The numerical results contain expanding wave and internal reflected waves and are close to those in \cite{HLZ19,ZC23,ZL24} concerning a similar example in the case of discontinuous $c(x)$ and $\sigma\equiv 1$.

\par The paper is organized as follows. 
In Section \ref{sec:ibvp and scheme}, an initial-boundary value problem for the general multidimensional acoustic wave equation is formulated, and the several versions of the semi-explicit compact fourth order scheme to solve it are constructed.
Two propositions concerning the fourth order truncation error and the algebraic properties of a generalized Numerov operator are included as well.
A discussion of the stability condition is added too.
Section \ref{sec:numer experim} is devoted to three 2D numerical experiments.

\section{\normalsize An initial-boundary value problem for the general acoustic wave equation and the semi-explicit compact fourth-order scheme}
\label{sec:ibvp and scheme} 
\setcounter{equation}{0}
\setcounter{proposition}{0}
\setcounter{theorem}{0}
\setcounter{remark}{0}
\setcounter{corollary}{0}

We formulate the initial-boundary value problem (IBVP) for the $n$-dimensional general acoustic wave equation
\begin{gather}
\beta\partial_t^2u=L(\sigma)u+f(x,t),\ \ (x,t)\in Q=Q_T=\Omega\times (0,T),
\label{eq}\\
u(x,t)=g(x,t),\ \ (x,t)\in\Gamma_T=\partial\Omega\times (0,T),
\label{bc}\\  u|_{t=0}=u_0(x),\ \ 
\partial_tu|_{t=0}=u_1(x),\ \  x\in \Omega:=(0,X_1)\times\ldots\times(0,X_n)
\label{ic}
\end{gather}
under the nonhomogeneous Dirichlet boundary condition, $n\geq 1$.
Here 
\[
\beta=\frac{1}{\sigma c^2},\ \ 
L(\sigma )=L_1(\sigma )+\ldots+L_n(\sigma ),\ \ L_k(\sigma)u:=\partial_k\Big(\frac{1}{\sigma}\partial_k u\Big),\ \ 1\leq k\leq n,
\]  
where $\sigma=\sigma (x)>0$ and $c=c(x)>0$ on $\bar{\Omega}$  are the density and speed of sound of the stationary medium, and $f(x,t)$ is a given source function;
recall also that $u$ is the pressure, for example, see \cite{BG12}.
We cover not only the standard cases $n=1,2,3$ since, in some problems in theoretical physics, wave equations for higher $n\geq 4$ are also of interest (for example, see \cite{F17}). 

\par We consider smooth solutions $u$ and reformulate the acoustic wave equation \eqref{eq} as the following system of equations containing only one second order derivative in time or space
\begin{gather}
\beta\partial_t^2u=u_1+\ldots+u_n+f,
\label{eq 1}\\
u_k=L_k(\sigma )u,\ \ 1\leq k\leq n.
\label{eq 2}
\end{gather}
Applying $\partial_t^2$ to the acoustic wave equation \eqref{eq} and using equation \eqref{eq 1}, we get
\begin{gather}
\beta\partial_t^4u=L(\sigma )
\partial_t^2u+\partial_t^2f=L(\sigma )
\Big[\frac{1}{\beta}(u_1+\ldots+u_n+f)\Big]+\partial_t^2f
\label{form for dt4}.
\end{gather}

\par Let $\bar{\omega}^{h_t}$ be the uniform mesh on $[0,T]$ with the nodes $t_m=mh_t$, $0\leq m\leq M$, and the step $h_t=T/M$, $M\geq 2$.
Let ${\omega}^{h_t}=\bar{\omega}^{h_t}\backslash\{0,T\}$ as well as $y^m=y(t_m)$, $\hat{y}^m=y^{m+1}$ and $\check{y}^m=y^{m-1}$.
Define the difference operators in $t$ 
\[
\delta_ty=\frac{\hat{y}-y}{h_t},\ \ 
\bar{\delta}_ty=\frac{y-\check{y}}{h_t},\ \
\Lambda_ty=\delta_t\bar{\delta}_ty=\frac{\hat{y}-2y+\check{y}}{h_t^2}.
\]

\par Applying the well-known expansion of $\Lambda_tu$, equation \eqref{eq 1} and formula \eqref{form for dt4}, we obtain
\begin{gather}
\Lambda_tu=\partial_t^2y+\frac{h_t^2}{12}\partial_t^4y+\mathcal{O}(h_t^4)
 =\frac{1}{\beta}(u_1+\ldots+u_n+f)
\nonumber\\
 +\frac{1}{\beta}\frac{h_t^2}{12}L(\sigma )
 \Big[\frac{1}{\beta}(u_1+\ldots+u_n+f)\Big]+\frac{h_t^2}{12}\frac{1}{\beta}\Lambda_tf+\mathcal{O}(h_t^4)
\label{eq with Lambda t} 
\end{gather}
on $\omega_{h_t}$, where $\partial_t^2f$ has been replaced with $\Lambda_tf$ with the reminder of the same order to avoid usage of derivatives of $f$ in $t$.

\par Let $y_0=y|_{t=0}$ for any function $y=y(t)$ and $I$ be the identity  operator.
Applying the Taylor formula at $t=0$, equations \eqref{eq}, \eqref{eq 1} and \eqref{eq 2} and the initial conditions \eqref{ic}, we obtain
\begin{gather}
 (\delta_tu)^0=u_1+\frac{h_t}{2}(\partial_t^2u)_0+\frac{h_t^2}{6}(\partial_t^3u)_0+\frac{h_t^3}{24}(\partial_t^4u)_0+\mathcal{O}(h_t^4)
=u_1+\frac{h_t}{2}\frac{1}{\beta}(u_{10}+\ldots+u_{n0}+f_0)
\nonumber\\
+\frac{h_t^2}{6}\frac{1}{\beta}\big(L(\sigma )u_1+(\partial_tf)_0\big)
+\frac{h_t^3}{24}\frac{1}{\beta}\Big\{L(\sigma )\Big[\frac{1}{\beta}(u_{10}+\ldots+u_{n0}+f_0)\Big]+(\partial_t^2f)_0\Big\}+\mathcal{O}(h_t^4)
\nonumber\\
=u_1+\frac{h_t^2}{6}\frac{1}{\beta}L(\sigma )u_1
+\frac{h_t}{2}\Big\{
\Big(I+\frac{h_t^2}{12}\frac{1}{\beta}L(\sigma )\Big)\Big[\frac{1}{\beta}(u_{10}+\ldots+u_{n0}+f_0)\Big]
+\frac23\frac{1}{\beta}\big(f|_{t=\frac{h_t}{2}}-f_0\big)\Big\}
\nonumber\\
+\mathcal{O}(h_t^4),
\label{2nd ini cond}
\end{gather}
where the derivatives of $f$ in $t$ have been excluded using the following formula 
from \cite{ZK21}:
\[
\frac23\big(f|_{t=\frac{h_t}{2}}-f_0\big)=\frac{h_t}{
3}(\partial_tf)_0 + \frac{h_t^2}{12}
(\partial_t^2f)_0+\mathcal{O}(h_t^3).
\] 

\par Let $\bar{\omega}_{hk}$ be the uniform mesh in $x_k$ on $[0,X_k]$ with the nodes $x_{ki}=ih_k$, $0\leq i\leq N_k$, $N_k\geq 2$, and the step $h_k=X_k/N_k$.
Let $\omega_{hk}=\bar{\omega}_{hk}\backslash\{0,X_k\}$.

\par Introduce the rectangular mesh
$\bar{\omega}_{\*h}=\bar{\omega}_{h1}\times\ldots\times\bar{\omega}_{hn}$ in $\bar{\Omega}$, with the nodes $x_{\*i}=(x_{1i_1},\ldots,x_{ni_n})$ $=x_{1i_1}\*e_1+\ldots+x_{ni_n}\*e_n$, where $\*h=(h_1,\ldots,h_n)$, $\*i=(i_1,\ldots,i_n)$ and $\*e_1,\ldots,\*e_n$ is the canonical basis in $\mathbb{R}^n$. 
Let ${\omega}_{\*h}={\omega}_{h1}\times\ldots\times{\omega}_{hn}$ and $\partial\omega_{\*h}=\bar{\omega}_{\*h}\backslash{\omega}_{\*h}$ be the corresponding meshes in $\Omega$ and on $\partial\Omega$ as well as $w_{\*i}=w(x_{\*i})$ and $w_{\*i-0.5\*e_k}=w(x_{\*i}-0.5h_k\*e_k)$.

\par Let $k=1,\ldots,n$. 
We define the two difference operators
\[
\Lambda_k(\sigma )w_{\*i}=\frac{1}{h_k}\Big(\frac{w_{\*i+\*e_k}-w_{\*i}}{\hat{\sigma }_{\*i+0.5\*e_k}h_k}-\frac{w_{\*i}-w_{\*i-\*e_k}}{\hat{\sigma }_{\*i-0.5\*e_k}h_k}\Big),\ \ s_k(\sigma )w_{\*i}:=\frac{1}{\sigma _{\*i}^{(k)}}w_{\*i}+\frac{h_k^2}{12}\Lambda_k(\sigma )w_{\*i}.
\]
Clearly, the more explicit formula holds
\begin{gather}
s_k(\sigma )w_{\*i}=
\alpha_{\*i}^{(k)}w_{\*i-\*e_k}
+\beta_{\*i}^{(k)}w_{\*i}+\alpha_{\*i+\*e_k}^{(k)}w_{\*i+\*e_k},
\label{oper sk 1}
\end{gather}
with the coefficients
\begin{gather}\alpha_{\*i}^{(k)}=\frac{1}{12\hat{\sigma }_{\*i-0.5\*e_k}},\ \ 
\beta_{\*i}^{(k)}=\frac{1}{\sigma _{\*i}^{(k)}}-\alpha_{\*i}^{(k)}-\alpha_{\*i+\*e_k}^{(k)}.
\label{oper sk 2}
\end{gather}
Here $\hat{\sigma }_{\*i-0.5\*e_k}=\hat{\sigma }_{I(\*i-0.5\*e_k)}$ is the following mean value for $k=1,\ldots,n$, respectively,
\begin{gather}
\hat{\sigma }_{I(\*i-0.5\*e_1)}=\frac{1}{h_1}\int_{x_{1(i_1-1)}}^{x_{1i_1}}
\sigma (x_1,x_{\*i(1)})\,dx_1,\ldots,
\hat{\sigma }_{I(\*i-0.5\*e_n)}=\frac{1}{h_n}\int_{x_{n(i_n-1)}}^{x_{ni_n}}
\sigma (x_{\*i(n)},x_n)\,dx_n
\label{d mean value}
\end{gather}
with $x_{\*i(1)}=(x_{2i_2},\ldots,x_{ni_n}),\ldots,
x_{\*i(n)}=(x_{1i_1},\ldots,x_{(n-1)i_{n-1}})$,
or $\hat{\sigma }_{\*i-0.5\*e_k}=\hat{\sigma }_{S(\*i-0.5\*e_k)}$ or $\hat{\sigma }_{G(\*i-0.5\*e_k)}$ are the related fourth-order Simpson and Gauss (with two nodes) scaled quadrature formulas
\begin{gather}
 \hat{\sigma }_{S(\*i-0.5\*e_k)}:=\frac{1}{6}\big(\sigma (x_{\*i-\*e_k})+4\sigma (x_{\*i-0.5\*e_k})+\sigma (x_{\*i})\big)=\hat{\sigma }_{I(\*i-0.5\*e_k)}+\mathcal{O}(h_k^4),
\label{d simpson}\\
 \hat{\sigma }_{G(\*i-0.5\*e_k)}:= 
 \frac12\big(\sigma (x_{\*i-0.5\*e_k}-\theta_G h_k\*e_k)+\sigma (x_{\*i-0.5\*e_k}+\theta_G h_k\*e_k)\big)=\hat{\sigma }_{I(\*i-0.5\*e_k)}+\mathcal{O}(h_k^4)
\label{d gauss}
\end{gather}
with $\theta_G=\frac{1}{2\sqrt{3}}$.
We also consider two cases
\begin{gather}
\sigma_{\*i}^{(k)}=\sigma _{\*i},\ \ 
\sigma_{\*i}^{(k)}=\tilde{\sigma }_{\*i}^{(k)}:=\Big[\frac12\Big(\frac{1}{\hat{\sigma }_{\*i-0.5\*e_k}}+\frac{1}{\hat{\sigma }_{\*i+0.5\*e_k}}\Big)\Big]^{-1}.
\label{sigmas}
\end{gather}
We comment on the respective properties of the operator $s_k$ in Proposition  \ref{lem:about s_k} below.

\par Let $\Lambda(\sigma )=\Lambda_1(\sigma )+\ldots+\Lambda_n(\sigma )$. For functions $w(x)$ and $\sigma (x)$ smooth in $x_k$ on $\bar{\Omega}$ and $1\leq k\leq n$, the following truncation errors hold
\begin{gather}
L_k(\sigma )w-\Lambda_k(\sigma )w=\mathcal{O}(h_k^2), 
\label{trunc err 1}\\
\Lambda_k(\sigma )w-s_k(\sigma )(\sigma ^{(k)}w)=\mathcal{O}(h_k^4)
\label{trunc err 2}
\end{gather}
on $\omega_{\*h}$.
Here, for $\sigma ^{(k)}=\tilde{\sigma }^{(k)}$, we assume that $\sigma (x)$ is given and smooth in $x_k$ on $\bar{\Omega}^{(k)}$ that enlarges $\bar{\Omega}$ by replacing $[0,X_k]$ with $[-X_k,2X_k]$.
Formula \eqref{trunc err 1} is well-known, for example, see \cite{SA79};
concerning formula  \eqref{trunc err 2}, see Proposition \ref{lem:4th order for s_k} below.
Then we can pass from formula \eqref {eq with Lambda t} and equation \eqref{eq 2} to 
\begin{gather}
\Lambda_tu
 =\Big(I+\frac{1}{\beta}\frac{h_t^2}{12}\Lambda(\sigma)
 \Big)\Big[\frac{1}{\beta}(u_1+\ldots+u_n+f)\Big]+\frac{1}{\beta}\frac{h_t^2}{12}\Lambda_tf+\mathcal{O}(|\*h|^4+h_t^4),
\label{eq 1 with Lambda's}\\ 
s_k(\sigma )(\sigma ^{(k)}u_k)=\Lambda_k(\sigma )u+\mathcal{O}(h_k^4),\ \ 1\leq k\leq n, 
\label{eq 2 with Lambda's}
\end{gather}
on $\omega_{\*h}\times\omega_{h_t}$
and $\omega_{\*h}\times\bar{\omega}_{h_t}$, respectively.

\par We omit the remainders in formulas \eqref{eq 1 with Lambda's}--\eqref{eq 2 with Lambda's} and consider the main approximate solution $v\approx u$ and auxiliary functions $v_1\approx u_1,\ldots,v_n\approx u_n$ defined on $\bar{\omega}_{\*h}\times\bar{\omega}_{h_t}$ and satisfying the equations
\begin{gather}
\Lambda_tv
 =\Big(I+\frac{1}{\beta}\frac{h_t^2}{12}\Lambda(\sigma)
 \Big)\Big[\frac{1}{\beta}(v_1+\ldots+v_n+f)\Big]+\frac{1}{\beta}\frac{h_t^2}{12}\Lambda_tf,
\label{eq 1 app}\\ 
  s_k(\sigma )(\sigma ^{(k)}v_k)= \Lambda_k(\sigma )v
\label{eq 2 app}
\end{gather}
both valid on $\omega_{\*h}\times{\omega}_{h_t}$.
Here clearly $\frac{h_t^2}{12}\Lambda_tf=\frac{1}{12}(\hat{f}-2f+\check{f})$.
We supplement these equations with the boundary conditions
\begin{gather}
v|_{\partial\omega_{\*h}}=g,\ \ 
v_k|_{\partial\omega_{\*h}}=g_k,\ \ 1\leq k\leq n,
\label{bc approx}
\end{gather}
where in accordance with the acoustic wave equation \eqref{eq} and the boundary condition \eqref{bc} we have
\begin{gather*}
 g_k:=
\begin{cases} 
 \beta\partial_t^2g-\sum_{1\leq l\leq n,\,l\neq k}L_l(\sigma )g-f\ \ \text{for}\ \ x_k=0,X_k,
\\[1mm] 
 L_k(\sigma )g\ \ \text{for}\ \ x_l=0,X_l,\ \ 1\leq l\leq n,\ \ l\neq k.  
\end{cases}
\end{gather*}

\par Formulas \eqref{eq 1 with Lambda's} and \eqref{eq 2 with Lambda's} demonstrate that the truncation errors of equations \eqref{eq 1 app} and \eqref{eq 2 app} are of the fourth orders $\mathcal{O}(|\*h|^4+h_t^4)$ and $\mathcal{O}(h_k^4)$;
the truncation error of the boundary conditions \eqref{bc approx} equals 0.

\par Using formula \eqref{trunc err 1} in expansion \eqref{2nd ini cond} as well, omitting the arising reminder $\mathcal{O}(|\*h|^4+h_t^4)$ and considering equation \eqref{eq 2 app} for $m=0$, 
we obtain the initial conditions for the scheme
\begin{gather}
v^0=u_0\ \ \text{on}\ \ \bar{\omega}_{\*h},
\label{1nd ini cond appr}\\
 (\delta_tv)^0
=u_1+\frac{h_t^2}{6}\frac{1}{\beta}\Lambda(\sigma )u_1
+\frac{h_t}{2}\Big\{
\Big(I+\frac{h_t^2}{12}\frac{1}{\beta}\Lambda(\sigma )\Big)\Big[\frac{1}{\beta}(v_1^0+\ldots+v_n^0+f_0)\Big]
\nonumber\\
+\frac23\frac{1}{\beta}\big(f|_{t=\frac{h_t}{2}}-f_0\big)\Big\}\ \ \text{on}\ \ {\omega}_{\*h},
\label{2nd ini cond appr}\\
s_k(\sigma )(\sigma ^{(k)}v_{k}^0)=\Lambda_k(\sigma )u_0,\ \ 1\leq k\leq n,\ \ \text{on}\ \ {\omega}_{\*h}. 
\label{2nd ini cond appr 2}
\end{gather}
We emphasize that these initial conditions do not contain derivatives of the data of the IBVP that allows one to apply them for nonsmooth data like in \cite{ZL25}. Similarly to equations \eqref{eq 1 app} and \eqref{eq 2 app}, the truncation errors of equations \eqref{2nd ini cond appr} and \eqref{2nd ini cond appr 2} are of the fourth orders $\mathcal{O}(|\*h|^4+h_t^4)$ and $\mathcal{O}(h_k^4)$.

\par The constructed scheme can be implemented easily.
For each $k=1,\ldots,n$, equations \eqref{2nd ini cond appr 2}
and \eqref{eq 2 app} together with the boundary conditions $(v_k^m-g_k^m)_{\*i}|_{i_k=0,N_k}=0$ lead to tridiagonal systems of linear algebraic equations for $\sigma ^{(k)}v_k^m$ in the direction $x_k$ and for time levels $m=0$ and $m=1,\ldots,M-1$ (the values for $m=M$ are not in use) except for the given values $v_{k,\*i}^m=g_{k,\*i}^m$ at the nodes on the facets (sides for $n=2$) $x_l=0,X_l$, $1\leq l\leq n$, $l\neq k$ of $\bar{\Omega}$.
The values of $v_k$ at the nodes on the edges (at the vertices for $n=2$) of $\bar{\Omega}$ are not in use.
Since $v^1=v^0+h_t(\delta_t v)^0$ and $\hat{v}=2v-\check{v}+h_t^2\Lambda_tv$, equations \eqref{2nd ini cond appr} and \eqref{eq 1 app} lead to explicit formulas for $v^{m+1}$ on $\omega_{\*h}$ for time levels $m=0$ and $m=1,\ldots,M-1$ provided that $\frac{1}{\beta}(v_1^m+\ldots+v_n^m)$ is already found.
\begin{remark}
 For some applications (including possible change of variables), the case of more general acoustic wave equation \eqref{eq} is of interest, with the operators $L(\sigma )$ and $L_k(\sigma )$ replaced  with $L(\*\sigma )=L_1(\sigma _1)+\ldots+L_n(\sigma _n)$ and $L_k(\sigma _k)$, where $\*\sigma =(\sigma _1,\ldots,\sigma _n)$ and $\sigma _k(x)>0$ on $\bar{\Omega}$, $1\leq k\leq n$. 
 The constructed compact scheme is generalized automatically to this case, with the mesh operators $\Lambda(\sigma )$, $\Lambda_k(\sigma )$ and $s_k(\sigma )$ replaced with $\Lambda(\*\sigma )=\Lambda_1(\sigma _1)+\ldots+\Lambda_n(\sigma _n)$, $\Lambda_k(\sigma _k)$ and $s_k(\sigma _k)$, respectively, as well as well as $\sigma ^{(k)}$ replaced with $\sigma _k^{(k)}$ in equations \eqref{eq 2 app} and \eqref{2nd ini cond appr 2}, $1\leq k\leq n$. 
\end{remark}

\begin{proposition} 
\label{lem:4th order for s_k}
Formula \eqref{trunc err 2} is valid, where, in the case $\sigma ^{(k)}=\tilde{\sigma }^{(k)}$, it is assumed that $\sigma (x)$ is given and smooth in $x_k$ on $\bar{\Omega}^{(k)}$.
\end{proposition}
\begin{proof}
\par 1. It is sufficient to consider the 1d case ($n=1$). For the ODE $L_1(\sigma )w=f(x)$ on $(0,X_1)$, the scheme 
\begin{gather}
\Lambda_1(\sigma )v=f+\frac{h_1^2}{12}\Lambda_1(\sigma )(\sigma f)\ \ \text{on}\ \ \omega_{1h},
\label{samara scheme}
\end{gather}
with $\hat{\sigma }=\hat{\sigma }_S$ was suggested in \cite{S63}, where
its fourth order truncation error $\Lambda_1(\sigma )w-f-\frac{h_1^2}{12}\Lambda_1(\sigma )(\sigma f)=\mathcal{O}(h_1^4)$ was proved (see also \cite{SA79}). The proof remains valid for $\hat{\sigma }=\hat{\sigma }_I$ and $\hat{\sigma }_G$ as well.
This justifies formula  \eqref{trunc err 2} in the case $\sigma ^{(k)}=\sigma $.

\par 2. Consequently, in the case $\sigma ^{(k)}=\tilde{\sigma }^{(k)}$, it is sufficient to prove that $r:=\tilde{\sigma }-\sigma $ satisfies the bound
\begin{gather}
 h^2\Lambda_1(\sigma )(rf)=\mathcal{O}(h^4)\ \ \text{on}\ \ \omega_{1h}. 
\label{aux est 1}
\end{gather}
Using the Taylor formula $f_{i\pm 1}=f_i\pm h f_i'+\mathcal{O}(h^2)$, we have
\[
  h^2\Lambda_1(\sigma )(rf)_i=h^2(\Lambda_1(\sigma )r)_if_i
  +\Big(\frac{r_{i+1}}{\hat{\sigma }_{i+0.5}}-\frac{r_{i-1}}{\hat{\sigma }_{i-0.5}}\Big)hf_i'+(|r_{i-1}|+|r_{i+1}|)\mathcal{O}(h^2).
\]
Since $\hat{\sigma }_{i\pm 0.5}=\sigma _i+\mathcal{O}(h)$, we further get
\begin{gather}
  h^2\Lambda_1(\sigma )(rf)_i
\nonumber\\
=\mathcal{O}\big(|r_{i+1}-2r_i+r_{i-1}|+h(|r_{i+1}-r_i|+|r_i-r_{i-1}|)+h^2(|r_{i-1}|+|r_i|+|r_{i+1}|)\big).
\label{aux exp}
\end {gather}
Using the Taylor formula 
\[
\hat{\sigma }_{i\pm 0.5}=\sigma _i+\frac12\Big(\frac{h}{2}\Big)^2\sigma _i''\pm\Big(\frac{h}{2}\sigma _i'+\frac16\Big(\frac{h}{2}\Big)^3\sigma _i'''\Big)+\mathcal{O}(h^4)
\]
for $\sigma \in C^4[-X_1,2X_1]$, we obtain
\[
r_i=\frac{\hat{\sigma }_{i-0.5}\hat{\sigma }_{i+0.5}}{\frac12(\hat{\sigma }_{i-0.5}+\hat{\sigma }_{i+0.5})}-\sigma _i
 =\frac{\sigma _i^2+\frac{h^2}{4}\big(\sigma _i\sigma _i''-(\sigma _i')^2\big)+\mathcal{O}(h^4)}
 {\sigma _i+\frac{h^2}{8}\sigma _i''+\mathcal{O}(h^4)}-\sigma _i=\frac{h^2}{4}\Big(\frac12\sigma _i''-\frac{(\sigma _i')^2}{\sigma _i}\Big)+\mathcal{O}(h^4)
\]
on $\bar{\omega}_h$.
Inserting this expansion into formula \eqref{aux exp}, we derive bound  \eqref{aux est 1}.   
\end{proof}
Note that, in the case $\sigma (x)={\rm const}$, scheme \eqref{samara scheme} is reduced to the well-known Numerov scheme
\[\frac{v_{i+1}-2v_i+v_{i-1}}{\sigma h^2}=s_Nf_i:=\frac{1}{12}(f_{i-1}+10f_i+f_{i+1}).
\]
\par Define the Euclidean space $H_{\*h}$ of functions $w$ given on $\bar{\omega}_{\*h}$, with $w|_{\partial\omega_{\*h}}=0$, endowed with the inner product $(w,z)_{H_{\*h}}=\sum_{x_{\*i}\in{\omega}_{\*h}}w_{\*i}z_{\*i}h_1\ldots h_n$.
\begin{proposition} 
\label{lem:about s_k}
Let $1\leq k\leq n$. The operator $s_k$ is self-adjoint in $H_{\*h}$.
\par 1. For $\sigma ^{(k)}=\sigma $, the operator $s_k$ is non-singular in $H_{\*h}$ provided that 
\begin{equation}
\frac{2-\delta_{i_k,1}}{12}\frac{\sigma _{\*i}}{\hat{\sigma }_{\*i-0.5\*e_k}}
 +\frac{2-\delta_{i_k,N_k-1}}{12}\frac{\sigma _{\*i}}{\hat{\sigma }_{\*i+0.5\*e_k}}\leq 1\ \ \text{for any}\ \ x_{\*i}\in\omega_{\*h},
\label{cond for pos def}
\end{equation}
and the inequality is strict for at least one value of $i_k=1,\ldots,N_k-1$,
where $\delta_{i,j}$ is the Kronecker symbol.
\par 2. For $\sigma ^{(k)}=\tilde{\sigma }^{(k)}$, we have $\beta_{\*i}^{(k)}=5\big(\alpha_{\*i}^{(k)}+\alpha_{\*i+\*e_k}^{(k)}\big)$ in \eqref{oper sk 2}, and consequently the operator $s_k$ is positive definite (thus, non-singular) in $H_{\*h}$.
\end{proposition}
\begin{proof}
The self-adjointness of $s_k$ follows from formula \eqref{oper sk 1}.
\par Inequality \eqref{cond for pos def} is equivalent to $\beta_{\*i}^{(k)}\geq(2-\delta_{i_k,1})\alpha_{\*i}^{(k)}+(2-\delta_{i_k,N_k-1})\alpha_{\*i+\*e_k}^{(k)}$.
The result of Item 1 follows from the well-known Taussky theorem concerning tridiagonal matrices.
Note that if $\hat{\sigma }=\hat{\sigma }_S$, then $6\hat{\sigma }_{\*i\pm0.5\*e_k}>\sigma _{\*i}$ and thus $\beta_{\*i}^{(k)}>0$.
\par Item 2 is elementary. It implies the diagonal dominance of $s_k$.
\end{proof}

\par Note that, for $\sigma ^{(k)}=\sigma $, if inequality \eqref{cond for pos def} is strict, then 
the operator $s_k$ is positive definite in $H_{\*h}$.
In addition, inequality \eqref{cond for pos def} is valid provided that 
\[
\frac13\leq\frac{\hat{\sigma }_{\*i\pm0.5\*e_k}}{\sigma _{\*i}},\ \ \text{or}\ \
\frac13\leq\frac{1}{\sigma _{\*i}}\min\limits_{x_{k(i_k-1)}\leq x_k\leq x_{k(i_k+1)}}\sigma ,\ \ \text{or}\ \ h_k\frac{1}{\sigma _{\*i}}\max\limits_{x_{k(i_k-1)}\leq x_k\leq x_{k(i_k+1)}}|\partial_k\sigma |\leq\frac23,
\]
for $x_{\*i}\in\omega_{\*h}$,
i.e., for a limited local range in values of $\sigma $ in $x_k$, or for sufficiently small $h_k$.
For $\sigma ^{(k)}=\tilde{\sigma }^{(k)}$, no such conditions are required that is an essential advantage of the latter choice.

\par For the constructed scheme, it can be expected according to the principle of frozen coefficients that the stability condition has the form
\begin{gather}
h_t^2\Big(\frac{1}{h_1^2}+\ldots+\frac{1}{h_n^2}\Big)\leq \ve\beta_{\min}\sigma _{\min}\ \ \text{with some}\ \ 0<\ve<\frac23,
\label{stab cond 1}
\end{gather}
where $0<\beta_{\min}\leq\beta(x)$ and $\sigma _{\min}\leq \sigma (x)$ on $\bar{\Omega}$, since, for stability in the  strong and standard energy norms with respect to the initial data and the free term and the corresponding error bounds of the orders $\mathcal{O}(|\*h|^{3.5})$ and $\mathcal{O}(|\*h|^4)$, condition \eqref{stab cond 1} in the case of variable $\beta(x)$ and  $\sigma (x)={\rm const}$ has recently been proved in \cite{ZL24,ZL25}. 

\par Clearly $\beta_{\min}\geq1/({\sigma _{\max}c_{\max}^2})$, where 
$c(x)\leq c_{\max}$ and $\sigma (x)\leq \sigma _{\max}$ on $\bar{\Omega}$, thus, a simpler though more restrictive stability condition takes the form
\begin{gather}
\nu_{\*h}^2(c,\sigma)= \frac{\sigma _{\max}}{\sigma _{\min}}\nu_{\*h}^2(c)\leq \ve\ \ 
\text{with}\ \ \nu_{\*h}^2(c):=
c_{\max}^2h_t^2\Big(\frac{1}{h_1^2}+\ldots+\frac{1}{h_n^2}\Big),\ \ 
0<\ve<\frac23,
\label{stab cond 2}
\end{gather}
where $\nu_{\*h}(c,\sigma)>0$ and $\nu_{\*h}(c)>0$ are the Courant numbers depending on both $c$ and $\sigma$ and only on $c$.
A similar stability condition was discussed in \cite{LLL23}.
In it, the presence of the spread of values $\sigma_{\max}/\sigma_{\min}$ is not so surprising since the acoustic wave equation \eqref{eq} can be rewritten as $\partial_t^2u=c^2\sigma L(\sigma )u+(c^2\sigma )f$, thus, the change  $\sigma \to\alpha \sigma $, with any $\alpha={\rm const}>0$, leads to the change $f\to\alpha f$ only. 
Similarly, this change
in $\sigma $ leads to the changes $v_k\to\alpha^{-1}v_k$ in equations  \eqref{eq 2 app} and \eqref{2nd ini cond appr 2} for $v_k$, $1\leq k\leq n$, but then in the change $f\to\alpha f$ only in equations \eqref{eq 1 app} and \eqref{2nd ini cond appr} for $v$.
Fortunately, the practical stability conditions arising in computations can be much softer with respect to $\sigma$ than the above theoretical ones, see the next Section. 

\par The main obstacle to prove stability for variable $\sigma$ is that, after eliminating the auxiliary unknowns $v_1,\ldots,v_n$, the difference operators arising in the canonical form of the scheme  are not self-adjoint and cannot be simultaneously symmetrized, cf. \cite{ZL24,ZL25}, while, for difference schemes to solve the second order hyperbolic equations, the existing stability theory is not sufficiently general in this respect.

\section{\normalsize Numerical experiments}
\label{sec:numer experim}
\setcounter{equation}{0}
\setcounter{proposition}{0}
\setcounter{theorem}{0}
\setcounter{remark}{0}
\setcounter{corollary}{0}

In this Section, we present results of three 2D numerical experiments. The code is implemented in \textit{Python 3}, and the plots are drawn with the use of graphical libraries \textit{matplotlib.pyplot} and \textit{plotly.graph\_objects}.
In Examples 1 and 2, the exact solution is known, and we compute the mesh $C$-norm (the uniform norm) and mesh $C^{1,0}$ and $C^1$ seminorms of the error $\rho=u-v$ at $t=t_M=T$:
\begin{gather*}
\|\rho^M\|_{C(\bar{\omega}_h)}:=\max_{x_{\*i}\in\bar{\omega}_h}|\rho_{\*i}^M|,
\\  
|\rho^M|_{C^{1,0}(\bar{\omega}_h)}:=
\max_{k=1,2}\,\max_{x_{\*i}\in\bar{\omega}_h,\,x_k\neq 0}|\bar{\delta}_k\rho_{\*i}^M|,\ \
|\rho^M|_{C^1(\bar{\omega}_h)}:=
\max\big\{|\rho^M|_{C^{1,0}(\bar{\omega}_h)},\,\|\bar{\delta}_t\rho^M\|_{C(\bar{\omega}_h)}\big\},
\end{gather*}
where $\bar{\delta}_k\rho_{\*i}=(\rho_{\*i}-\rho_{\*i-\*e_k})/h_k$.
In all the Examples below, we take $N_1=N_2=N$ and $h_1=h_2=h$.

\smallskip\par \textbf{Example 1}. We first take $\Omega=(0,2)\times (0,2)$, $T=1.2$ and the smooth density and squared speed  of sound 
\[
\sigma (x)=e^{x_1+x_2},\ \ c^2(x)=[1+0.5(x_1^2+x_2^2)]^2.
\]
Note that the spreads in their values over $\bar{\Omega}$, i.e., $\sigma _{\max}/\sigma _{\min}\approx 54.60$ and $(c_{\max}/c _{\min})^2=25$, are high enough.
We choose rather standard exact solution 
\[
u(x_1,x_2,t)=\cos(\sqrt{2}t-x_1-x_2)
\]
of a travelling wave type and compute the data $u_0$, $u_1$, $f$ and $g$ according to it (note that all of them are not identically equal to zero).

\par We consider two versions of choosing the scheme parameters: ($A$)  $\hat{\sigma}=\hat{\sigma}_S$ and $\sigma_{\*i}^{(k)}=\sigma _{\*i}$; 
($B$)~$\hat{\sigma}=\hat{\sigma}_G$ and $\sigma_{\*i}^{(k)}=\tilde{\sigma }_{\*i}^{(k)}$, see formulas \eqref{d simpson}--\eqref{sigmas}.
For version $A$, the values of $\sigma$ at the nodes of the mesh $\bar{\omega}_{hk}$ and in the middle between adjacent nodes in direction $x_k$, $1\leq k\leq n$, are used.
On the contrary, for version $B$, the values of $\sigma$ at the listed points in $x_k$ are not involved.

\par We present errors, ratios of the sequential errors and practical convergence rates 
\[
e(N,M),\ \ r(N,M)=\frac{e(N,M)}{e(N/2,M/2)},\ p(N,M)=\log_2\frac{e(N,M)}{e(N/2,M/2)},
\]
respectively, in the $C(\bar{\omega}_h)$ norm as well as $C^{1,0}(\bar{\omega}_h)$ and $C^1(\bar{\omega}_h)$ seminorms at $t=t_M$.
In this Example, for the chosen values of $N$ and $M$, the Courant numbers are $\nu_{\*h}(c)\approx 1.0607>1$ and $\nu_{\*h}(c,d)\approx 7.8376\gg 1$ (see \eqref{stab cond 2}); nevertheless, computations are stable and demonstrate excellent error values.

\par For versions $A$ and $B$, the results are given in Tables \ref{Example_1_version_A} and \ref{Example_1_version_B}.
The original value $N=5$ is small and the corresponding step $h=0.4$ is rough.
Notice that both the versions demonstrate very small level of the errors even for rough meshes, the ratios of sequential errors are mainly close to 16 and the practical convergence rates are close to 4 in the all chosen norm and seminorms (more close as $N$ and $M$ grow).
Naturally, the errors $e_{C^{1,0}}$ and $e_{C^1}$ are larger than $e_C$, and also the initial values of $r_{C^{1,0}}$ and $r_{C^1}$ are less close to 16 than $r_C$.
The difference in the results between versions $A$ and $B$ is not so essential.
\begin{table}[ht]
\centering
\caption{\textbf{Example 1}. Errors, error ratios and practical convergence rates in the $C$ norm and $C^{1,0}$ and $C^1$ seminorms for version $A$ of the scheme parameters.}
\label{Example_1_version_A}
\vskip 3pt
\begin{tabular}{rrccccccccc}
\hline
\noalign{\smallskip}
$N$ & $M$ 
& $e_{C}$ & $r_{C}$ & $p_{C}$ 
& $e_{C^{1,0}}$& $r_{C^{1,0}}$ & $p_{C^{1,0}}$ 
& $e_{C^1}$ & $r_{C^1}$ & $p_{C^1}$  \\
\noalign{\smallskip}
\hline
\noalign{\smallskip}
5	&20	    &1.439e-4	&$-$	&$-$	&2.729e-4	&$-$	&$-$	&3.906e-4	&$-$	&$-$\\
10	&40	    &9.499e-6	&15.14	&3.921	&2.211e-5	&12.34	&3.626	&2.786e-5	&14.02	&3.809\\
20	&80	    &6.017e-7	&15.79	&3.981	&1.530e-6	&14.46	&3.854	&1.777e-6	&15.68	&3.971\\
40	&160	&3.779e-8	&15.92	&3.993	&9.943e-8	&15.38	&3.943	&1.113e-7	&15.96	&3.996\\
80	&320	&2.366e-9	&15.97	&3.998	&6.389e-9	&15.56	&3.960	&6.866e-9	&16.21	&4.019\\
\hline
\end{tabular}
\end{table}
\begin{table}[ht]
\centering
\caption{\textbf{Example 1}. Errors, error ratios and practical convergence rates in the $C$ norm and $C^{1,0}$ and $C^1$ seminorms for version $B$ of the scheme parameters.}
\label{Example_1_version_B}
\vskip 3pt
\begin{tabular}{rrccccccccc}
\hline
\noalign{\smallskip}
$N$ & $M$ 
& $e_{C}$ & $r_{C}$ & $p_{C}$ 
& $e_{C^{1,0}}$& $r_{C^{1,0}}$ & $p_{C^{1,0}}$ 
& $e_{C^1}$ & $r_{C^1}$ & $p_{C^1}$  \\
\noalign{\smallskip}
\hline
\noalign{\smallskip}
 5&20   &1.617e-4  &$-$    &$-$    &3.984e-4  &$-$	   &$-$	   &1.060e-3  &$-$    &$-$\\	
10&40	&1.119e-5	&14.45	&3.853	&3.265e-5	&12.20	&3.609	&7.493e-5	&14.14	&3.822\\	
20&80	&7.175e-7	&15.60	&3.964	&2.310e-6	&14.14	&3.821	&4.715e-6	&15.89	&3.990\\	
40&160	&4.482e-8	&16.01	&4.001	&1.539e-7	&15.01	&3.908	&2.935e-7	&16.07	&4.006\\	
80&320	&2.803e-9	&15.99	&3.999	&9.935e-9	&15.49	&3.953	&1.835e-8	&16.00	&4.000\\	
\hline
\end{tabular}
\end{table}

\smallskip\par \textbf{Example 2}. We take the same $\Omega$, $T$ and $u(x,t)$ but the different density and the squared speed  of sound 
\[
\sigma=\sigma(x_1)=[1.25+0.75\tanh(b_\sigma(x_1-1))]^{-1},\ \  c^2=c^2(x_1)=2.5-1.5\tanh(b_c(x_1-1))
\]
with the smoothed jumps at $x_1=1$. 
We set $b_\sigma=20$ and $b_c=1000$, and the jump in $c^2$  is very steep, see Figure \ref{graphs of sigma and c2}(a).
\begin{figure}[tbh!]
\begin{minipage}{0.5\textwidth}
\center{\includegraphics[width=1\linewidth]{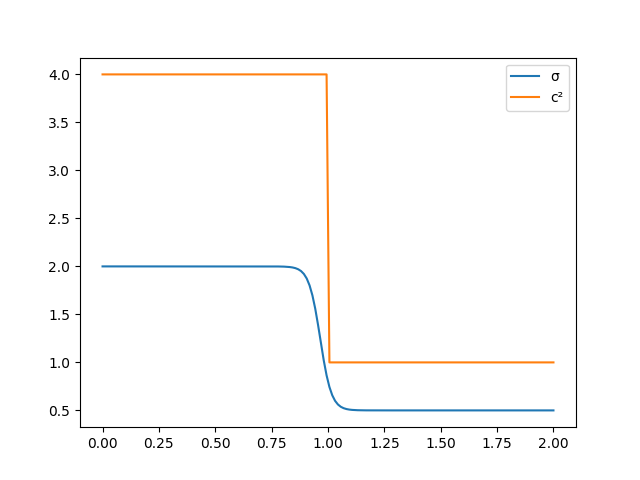}} (a) \\
\end{minipage}
\hfill
\begin{minipage}{0.5\textwidth}
\center{\includegraphics[width=1\linewidth]{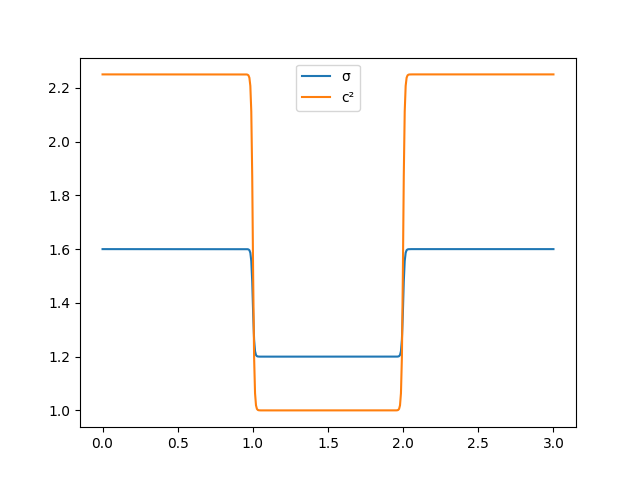}} (b) \\
\end{minipage}
\label{graphs of sigma and c2}
\caption{Graphs of $\sigma(x_1)$ and $c^2(x_1)$ involved in: (a) Example 2; (b) Example 3}
\end{figure}

\par In this Example, for our values of $N$ and $M$, $\nu_{\*h}(c)\approx 0.8485$ but $\nu_{\*h}(c,d)\approx 1.6970>1$; nevertheless, computations remain stable.
Once again we choose the above versions $A$ and $B$ of the scheme parameters. The respective numerical results are shown in Tables \ref{Example_2_version_A} and \ref{Example_2_version_B}.
The original value $N=10$ is taken too rough, and the corresponding errors are large, but the further behaviour of the errors as $N$ grows is interesting and different from Example 1. For the next value $N=20$, the practical convergence rates are higher than 2 though less than 3. 
But, for the further values $N=40$ and $80$, $p_C$ is much higher than 4 and close to 6, and it becomes very close to 4 for the next $N=160$.
Concerning $p_{C^{1,0}}$ and $p_{C^1}$, they are also higher than 4 for at least one of $N=40$ and $80$.

\par We emphasize that for the error behaviour, the rate of smoothing of $\sigma$ is definitive since $c^2$ is practically discontinuous for the chosen meshes.
The phenomenon of the 4th order error behaviour for discontinuous $c^2$ is not elementary at all and needs more theoretical investigation since, in formulas \eqref{eq 1 app} and \eqref{2nd ini cond appr} of the scheme, the second order difference operator $\Lambda(\sigma)$ is applied to the term including the multiplier $\frac{1}{\beta}=\sigma c^2$.

\par The difference in the results between versions $A$ and $B$ is not so significant once again though the behaviour of $r_C$, $e_C$, $r_{C^1}$ and $p_{C^1}$ is generally more regular for the latter version.
\begin{table}[ht]
\centering
\caption{\textbf{Example 2}. Errors, error ratios and practical convergence rates in the $C$ norm and $C^{1,0}$ and $C^1$ seminorms for version $A$ of the scheme parameters.}
\label{Example_2_version_A}
\vskip 3pt
\begin{tabular}{rrcccccccccc}
\hline
\noalign{\smallskip}
$N$ & $M$ 
& $e_{C}$ & $r_{C}$ & $p_{C}$ 
& $e_{C^{1,0}}$& $r_{C^{1,0}}$ & $p_{C^{1,0}}$ 
& $e_{C^1}$ & $r_{C^1}$ & $p_{C^1}$  \\
\noalign{\smallskip}
\hline
\noalign{\smallskip}
10&20  &4.171e-1&    $-$     &$-$   &$\hspace{-16pt}$1.161    &$-$&   $-$     &$\hspace{-16pt}$1.161   &$-$        &$-$\\
20&40  &6.885e-2	&6.058	&2.599	&2.236e-1	&5.193	&2.377	&2.296e-1	&5.058	&2.339\\
40&80  &1.149e-3	&59.93	&5.905	&1.101e-2	&20.31	&4.344	&1.784e-2	&12.87	&3.686\\
80&160 &1.346e-5	&85.37	&6.416	&5.355e-4	&20.56	&4.362	&8.262e-4	&21.59	&4.433\\
160&320&8.356e-7	&16.11	&4.009	&3.872e-5	&13.83	&3.790	&3.872e-5	&21.34	&4.415\\
\hline
\end{tabular}
\end{table}
\begin{table}[ht]
\centering
\caption{\textbf{Example 2}. Errors, error ratios and practical convergence rates in the $C$ norm and $C^{1,0}$ and $C^1$ seminorms for version $B$ of the scheme parameters.}
\label{Example_2_version_B}
\vskip 3pt
\begin{tabular}{rrcccccccccc}
\hline
\noalign{\smallskip}
$N$ & $M$ 
& $e_{C}$ & $r_{C}$ & $p_{C}$ 
& $e_{C^{1,0}}$& $r_{C^{1,0}}$ & $p_{C^{1,0}}$ 
& $e_{C^1}$ & $r_{C^1}$ & $p_{C^1}$  \\
\noalign{\smallskip}
\hline
\noalign{\smallskip}
10&20  &4.223e-1&    $-$     &$-$   &$\hspace{-16pt}$1.151    &$-$&   $-$     &$\hspace{-16pt}$1.151   &$-$        &$-$\\
20 &40	&6.858e-2	&6.157	&2.622	&2.202e-1	&5.229	&2.387	&2.202e-1	&5.229	&2.387\\
40 &80	&1.086e-3	&63.17	&5.981	&7.712e-3	&28.55	&4.836	&1.336e-2	&16.47	&4.042\\
80 &160	&1.984e-5	&54.73	&5.774	&6.047e-4	&12.75	&3.673	&6.554e-4	&20.39	&4.350\\
160&320	&1.229e-6	&16.15	&4.013	&4.544e-5	&13.31	&3.734	&4.544e-5	&14.42	&3.850\\
\hline
\end{tabular}
\end{table}

\smallskip\par \textbf{Example 3}. In this example, we study the acoustic wave propagation in the three-layer-type medium in $\Omega=(-0.7,3.7)\times (-0.7,3.7)$
for $T=1.15$
and take the density and squared speed of sound in the form
\begin{gather*}
\sigma=\sigma(x_1)=1.6-0.2\,[\tanh\big(b_\sigma(x_1-1)\big)-\tanh\big(b_\sigma(x_1-2)\big)],
\\
c^2=c^2(x_1)=2.25-0.625\,[\tanh\big(b_c(x_1-1))-\tanh\big(b_c(x_1-2)\big)],
\end{gather*}
with the smoothed jumps in the defining densities $\sigma =1.6$, 1.2 and 1.6  
and speeds of sound $c = 1.5$, 1 and 1.5
in the left $-0.7\leq x_1\leq 1$, middle $1\leq x_1\leq 2$ and right $2\leq x_1\leq 3.7$ layers in $x_1$, respectively. 
We also set $b_\sigma=b_c=100$ (physically, these parameters should be equal or close), so the jumps are steep, see Figure \ref{graphs of sigma and c2}(b) where $\sigma$ and $c^2$ are given for $0\leq x_1\leq 3$.

\par We also use the Ricker-type wavelet source function smoothed in space
\[
f(x_1,x_2)=\beta(x_1,x_2)\frac{\gamma}{\pi}e^{-\gamma((x_1-1.5)^2+(x_2-1.5)^2)}\sin(50t)e^{-200t^2},
\]
with $\gamma=1000$, where $(1.5,1.5)$ is the centre of $\Omega$.
Note that then $\frac{\gamma}{\pi}\approx 318.3$ and $\frac{\gamma}{\pi}e^{-\gamma r^2}<5.385\times 10^{-8}$ for $r>0.15$.
The other data are zero: $u_0=u_1=0$ and $g=0$.

\par For definiteness, we apply version A of the scheme parameters and choose $N=480$ and $M=460$; thus, $h=11/1200\approx 9.167\times 10^{-3}$ and $h_t=0.025$. For them, $\nu_{\*h}(c)\approx 0.8678$ and $\nu_{\*h}(c,d)\approx 1.002$; the computations are stable again.

\par Contour levels of wavefields at six sequential characteristic time moments are presented in Fig. \ref{contour figures}.
The corresponding perpendicular central sections of the wavefields, for $x_1=1.5$ and $x_2=1.5$, at four time moments are given in Fig. \ref{dyn at fixed x and y}.
We observe the wavefront generated by the source function expanding in the middle layer and then passing to the left and right layers with the higher speed of sound, together with the internal wavefronts reflected back from both the lines of jumps in $c^2(x)$ and $\sigma(x)$ towards the centre of $\Omega$.
The reflected waves meet at the centre and pass through each other.
The results are given in the same manner and are close in general to those presented in \cite{ZL25}, see also \cite{HLZ19,ZC23}, where similar but discontinuous $c(x)$ and $\sigma(x)\equiv 1$ were taken.
In addition, the 3D graphs of the wavefields at an intermediate and the final time moments are shown in Fig. \ref{dyn3d}. 
Such graphs are absent in \cite{HLZ19,ZC23,ZL24}, although they probably most evidently demonstrate the complex overall structure of the wavefields containing not only moving and reflected wavefronts but moving and incipient narrow peaks as well.
\begin{figure}[tbh!] 
\begin{minipage}{0.5\textwidth}
\center{\includegraphics[width=1\linewidth]{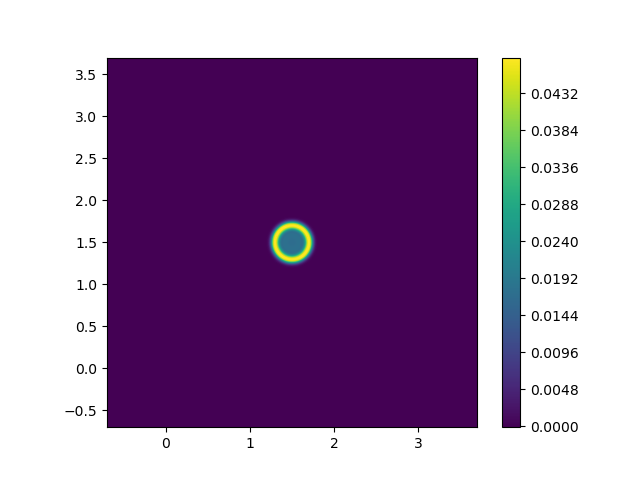}} (a) \\
\end{minipage}
 \hfill
\begin{minipage}{0.5\textwidth}
\center{\includegraphics[width=1\linewidth]{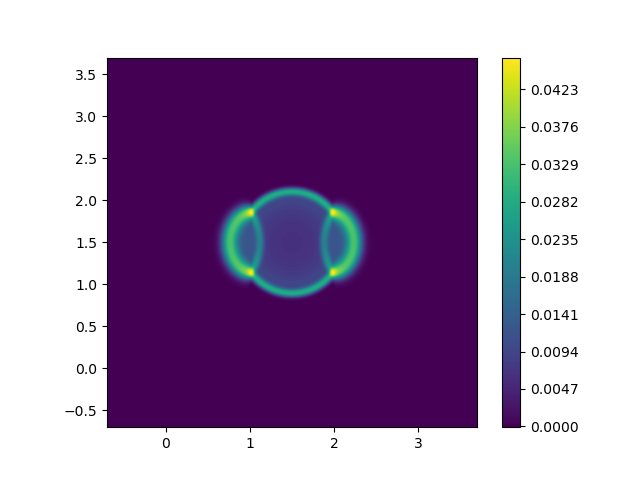}} (b) \\
\end{minipage}
 \vfill
 \begin{minipage}{0.5\textwidth}
\center{\includegraphics[width=1\linewidth]{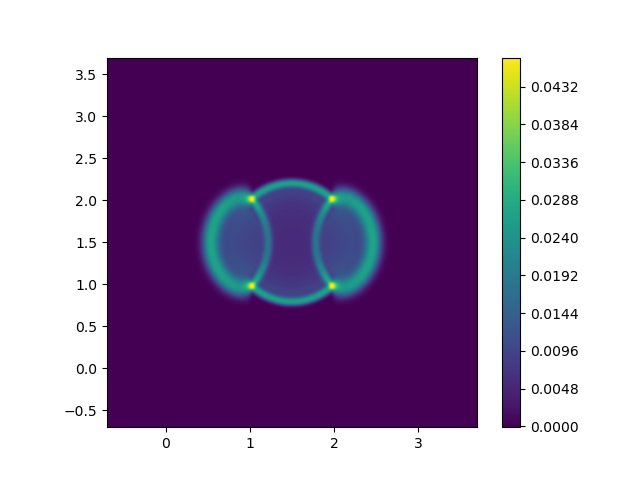}} (c) \\
\end{minipage}
 \hfill
\begin{minipage}{0.5\textwidth}
\center{\includegraphics[width=1\linewidth]{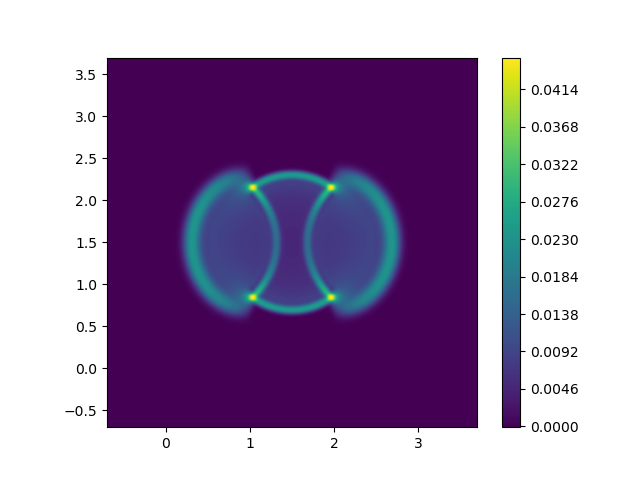}} (d) \\
\end{minipage}
 \vfill
\begin{minipage}{0.5\textwidth}
\center{\includegraphics[width=1\linewidth]{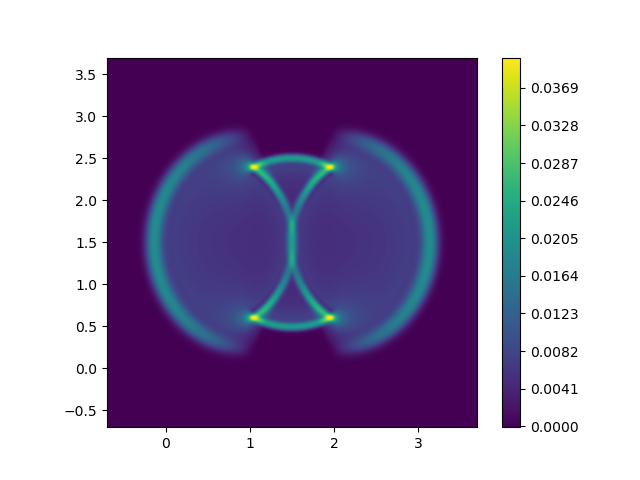}} (e) \\
\end{minipage}
 \hfill
\begin{minipage}{0.5\textwidth}
\center{\includegraphics[width=1\linewidth]{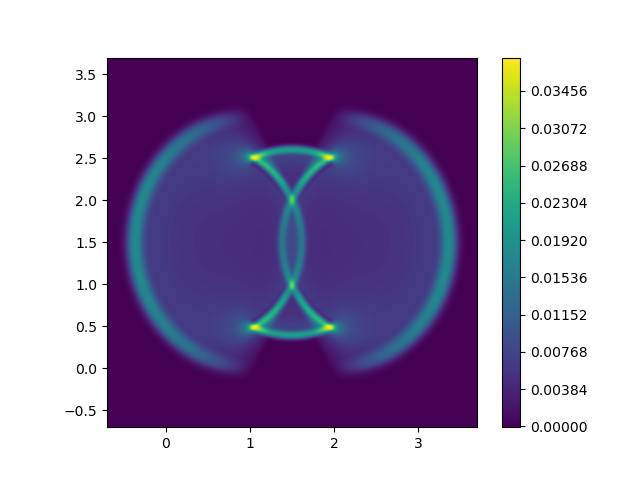}} (f) \\
\end{minipage}
\caption{Example 3. Contour levels of wavefields at several times: (a) $t=0.25$; (b) $t=0.65$; (c) $t=0.75$; (d) $t=0.85$; (e) $t=1.05$; (f) $=1.15$.}
\label{contour figures}
\end{figure}
\begin{figure}[tbh!]
\begin{minipage}{0.5\textwidth}
\center{\includegraphics[width=1\linewidth]{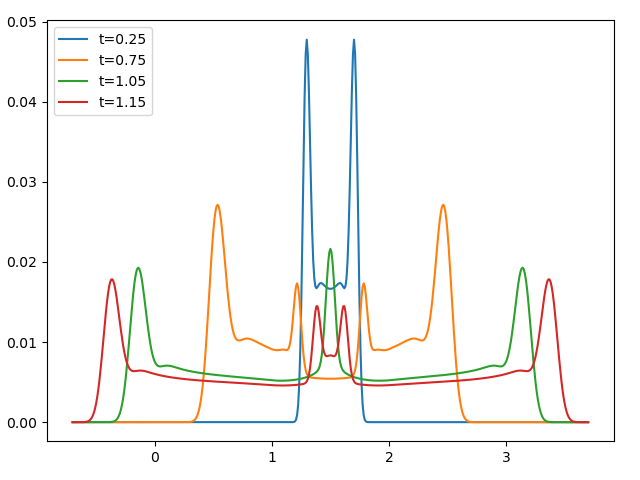}} (a) \\
 \end{minipage}
 \hfill
 \begin{minipage}{0.5\textwidth}
\center{\includegraphics[width=1\linewidth]{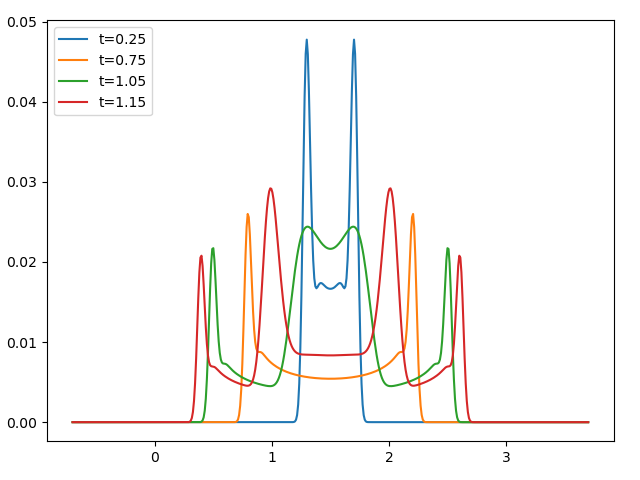}} (b) \\
\end{minipage}
\caption{Example 3. The perpendicular sections of the wavefields at several times: (a) for $x_2=1.5$; (b) for $x_1=1.5$.}
\label{dyn at fixed x and y}
\end{figure}
\begin{figure}[tbh!]
\begin{minipage}{0.5\textwidth}
\center{\includegraphics[width=1\linewidth]{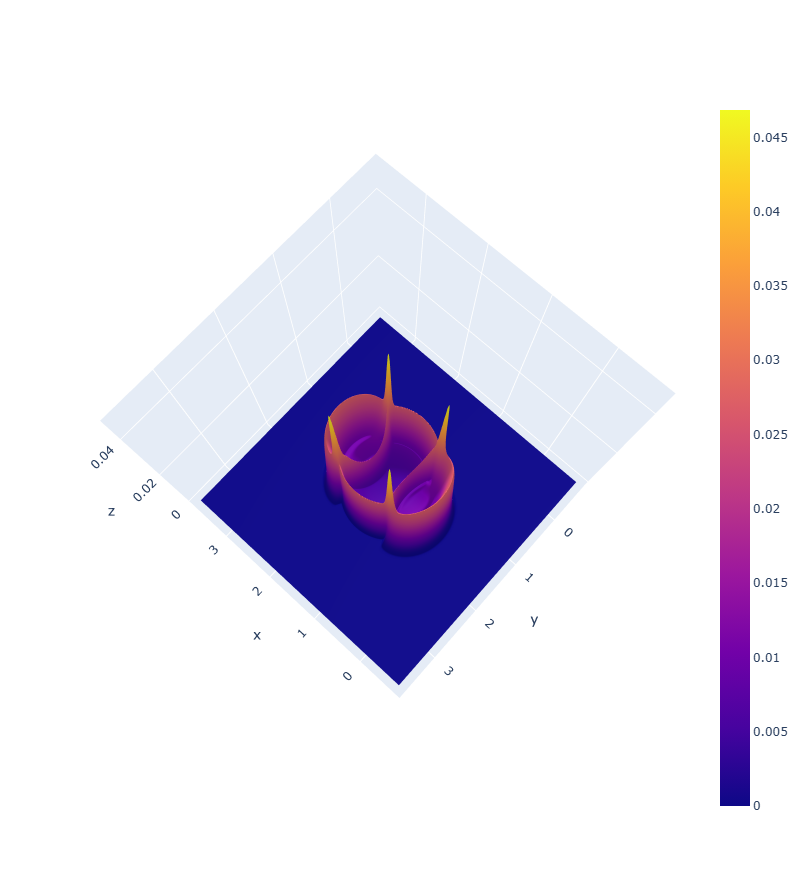}} (a) \\
\end{minipage}
\hfill
\begin{minipage}{0.5\textwidth}
\center{\includegraphics[width=1\linewidth]{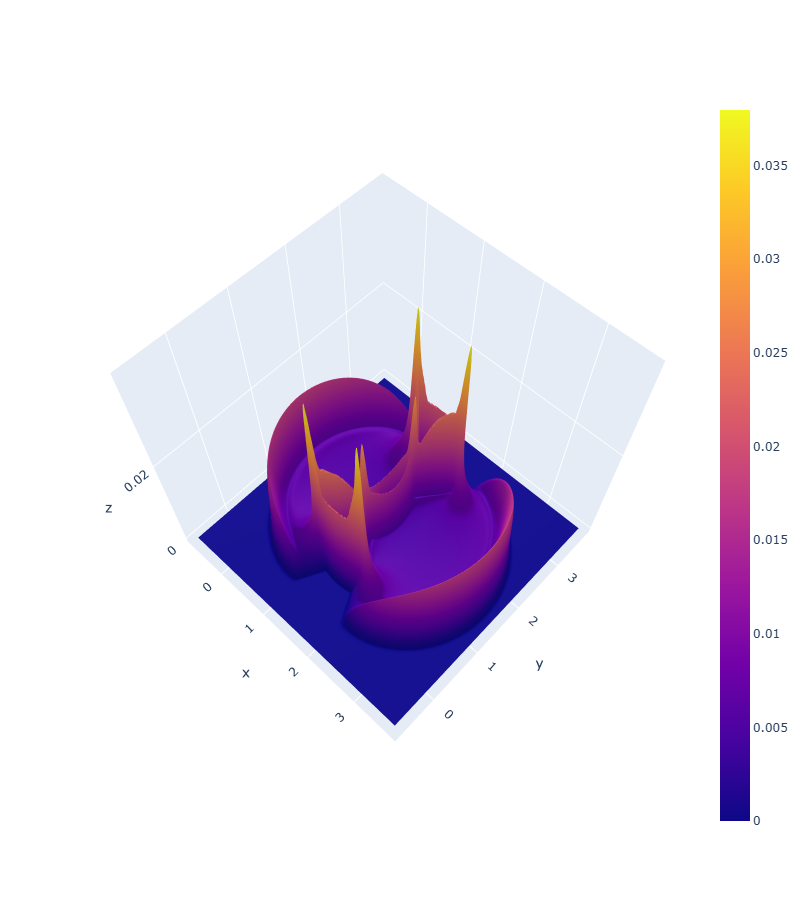}} (b) \\
\end{minipage}
\caption{Example 3. The 3D graphs of the wavefields at: (a) $t=0.75$; (b) $t=1.15$.}
\label{dyn3d}
\end{figure}
\section*{\normalsize Acknowledgments}

Support from the Basic Research Program at the HSE University (Laboratory of Mathematical Methods in Natural Science) is gratefully acknowledged by A. Zlotnik, Sections \ref{sec:intro}--\ref{sec:ibvp and scheme}. Support from the Basic Research Program at the HSE University (International Centre of Decision Choice and Analysis) is gratefully acknowledged by T. Lomonosov, Section~\ref{sec:numer experim}.

\renewcommand{\refname}{\normalsize\rm 
\end{document}